\documentclass[11pt]{amsart}

\usepackage{graphicx}
\usepackage{alltt}

\usepackage{graphics,color}

\usepackage{psfrag}

\newcommand{\Isom}{{\operatorname {Isom}}}
\newcommand{\QC}{{\bf {QC}}}
\newcommand{\QI}{{\bf {QI}}}

\newcommand{\D} {\mathbb {D}}
\newcommand{\Ha}{{\mathbb {H}}^{2}}
\newcommand{\Ho}{ {\mathbb {H}}^{3}}
\newcommand{\Hnn}{ {\mathbb {H}}^{n}}
\newcommand{\Bnn}{ {\mathbb {B}}^{3}}
\newcommand{\R} {\mathbb {R} }
\newcommand{\Z} {\mathbb {Z}}
\newcommand{\N}{\mathbb {N}}
\newcommand{\C} {\mathbb {C}}

\newcommand{\Sp} {\mathbb {S}}
\newcommand{\Hoo} {\mathbf {H}^3}
\newcommand{\Lin} {\mathcal{L}(\R^2)}

\newcommand{\pt} {\partial}

\newcommand{\lap} {\Delta}
\newcommand{\good} {\mathcal{G}}
\newcommand{\Ext} {\mathcal{E}}
\newcommand{\dist} {\mathbf{K}}
\newcommand{\dista} {\mathbf{d}}
\newcommand{\ener} {\mathbf{e}}
\newcommand{\Avr} {\mathbf{A}}
\newcommand{\green} {\mathbf{g}}

\newcommand{\const} {\operatorname{const}}

\newtheorem{theorem}{Theorem}[section]
\newtheorem{definition}{Definition}[section]
\newtheorem{lemma}{Lemma}[section]
\newtheorem{conjecture}{Conjecture}

\theoremstyle{remark}
\newtheorem*{remark}{Remark}

\begin{document}
\title[Harmonic maps between hyperbolic spaces ] {Harmonic maps between 3-dimensional hyperbolic spaces}
\author[Vlad Markovic]{Vladimir Markovic}

\address{\newline Department of Mathematics   \newline California Institute of Technology   \newline Pasadena, CA 91125,  USA}
\email{markovic@caltech.edu}

\today


\begin{abstract}  We prove that a quasiconformal map of the sphere $\Sp^{2}$ admits  a harmonic quasi-isometric extension to the hyperbolic  space $\Ho$, 
thus confirming the well known Schoen Conjecture in dimension 3. 
\end{abstract}

\maketitle

\let\johnny\thefootnote
\renewcommand{\thefootnote}{}

\footnotetext{Vladimir Markovic is supported by the NSF grant number DMS-1201463}
\let\thefootnote\johnny

\section{Introduction} 
\subsection{The Schoen Conjecture and the statements of the results}  
One of the main questions in the theory of harmonic maps  is when the homotopy
class of a map $F:M \to N$ between two negatively curved Riemannian manifolds
contains a harmonic map. When $F$ has finite total energy the theory is well
developed and the existence and uniqueness of the corresponding harmonic map has been 
established  (see the book \cite{s-y} by Schoen and Yau and  the article by Hamilton \cite{hamilton}).

In the case when $F$ does not have finite total energy much less is know. The case that has  mostly been studied  
is when $M$ and $N$ are both equal to the hyperbolic spaces $\Hnn$, $n \ge 2$, with the corresponding hyperbolic metrics. Here, we begin with a continuous  
map $f:\Sp^{n-1} \to \Sp^{n-1}$ and the question is whether $f$ can be (continuously) extended to a harmonic map
$H:\Hnn \to \Hnn$ (we make the usual identification $\partial{\Hnn} \equiv \Sp^{n-1}$). 

If no further assumptions on $f$ and  $H$ are imposed then this problem is too general.

\begin{remark} The corresponding Dirichlet problem for functions has been solved in the 80's by Anderson \cite{anderson} and Sullivan \cite{sullivan}. They proved that every  continuous map  $f:\Sp^{n-1} \to \R$ can be extended to a harmonic function
$H:\Hnn \to \R$. 
\end{remark}

The right level of generality was formulated by Schoen who posed the following conjecture.

\begin{conjecture}[Schoen, Li-Wang] Suppose that  $f:\Sp^{n-1} \to \Sp^{n-1}$ is a quasiconformal map. Then there exists a unique harmonic and quasi-isometric map $H:\Hnn \to \Hnn$ that extends $f$. 
\end{conjecture}

Schoen posed this conjecture in \cite{schoen} for $n=2$, but this was soon appropriately generalized \cite{l-w} to all hyperbolic spaces by Li and Wang. The uniqueness part  this conjecture was established by Li and Tam \cite{l-t-2} for $n=2$ and by Li and Wang \cite{l-w} for all $n$.

The existence part of the conjecture  has been much studied. It was shown by Li and Tam \cite{l-t-1} that if the boundary map 
$f:\Sp^{n-1} \to \Sp^{n-1}$ is a $C^1$ diffeomorphism then  $f$ admits a harmonic quasi-isometric extension to $\Hnn$ (every $C^1$ diffeomorphism of $\Sp^{n-1}$ is quasiconformal). The authors show that  one can use the heat flow method to prove the existence of harmonic maps in this case. Another important partial result was proved by Tam and Wan \cite{t-w},  and independently by Hardt and Wolf \cite{h-w}, and it says that if $f:\Sp^{n-1} \to \Sp^{n-1}$ admits a harmonic quasi-isometric extension that is also a quasiconformal self-map of $\Hnn$ then a small perturbation of $f$ admits a harmonic quasi-isometric extension too. 

We also mention the recent existence type result by Bonsante and Schlenker \cite{b-s} in which they prove that every quasisymmetric map $f:\Sp^1 \to \Sp^1$ has a unique  minimal Lagrangian quasiconformal extension to the hyperbolic disc $\Ha$ (minimal Lagrangian maps are close relatives of harmonic maps).

\begin{remark}
If we write $C=C(K_1,K_2,...)$, we mean that the  constant $C$ depends only on $K_1,K_2,...$. We use this policy throughout the paper.
\end{remark}

The main difficulty in working with harmonic maps between hyperbolic spaces is to control the harmonic map $H:\Hnn \to \Hnn$ inside the hyperbolic space in terms of the regularity of the boundary map  $f:\Sp^{n-1} \to \Sp^{n-1}$. To illustrate the subtlety of this issue we show in the next subsection that it is easy to construct a sequence of diffeomorhisms $f_n:\Sp^1 \to \Sp^1$, that converge to the identity in the $C^0$ sense, but such that the corresponding harmonic extensions degenerate on  compact sets in $\Ha$ and we can not extract any sort of limit (this behavior is very different for the harmonic functions problem we mentioned before that was solved by Anderson and Sullivan, where the boundary map effectively controls the behavior of the harmonic function inside the disc).

In the remainder of this paper we prove the following theorem  (see  the next section for the definition of a $(L,A)$-quasi-isometry) which takes care of the main difficulty we described above. 

\begin{theorem}\label{thm-main-1}  There exist  constants $L=L(K)>0$ and $A=A(K)\ge 0$, such that if  a $K$-quasiconformal map $f:\Sp^{2}\to \Sp^{2}$ has a harmonic quasi-isometric extension $H(f):\Ho \to \Ho$, then $H(f)$ is a $(L,A)$-quasi-isometry.
\end{theorem}

\begin{remark} One of the reasons behind the proof of Theorem \ref{thm-main-1} is the quasiconformal rigidity (in the sense of Mostow) that holds in higher dimensions. One way of expressing this is that every quasiconformal map of $\Sp^{n-1}$, $n \ge 3$, is differentiable almost everywhere (and the derivative has maximal rank).  It is very likely that the previous theorem can be extended to higher dimensional hyperbolic spaces. However, our method does not appear to tell us much in the case $n=2$, as quasiconformal maps of $\Sp^1$ (also known as quasisymmetric) may not be differentiable anywhere and are not rigid in this sense (one may be able to prove a similar result for bi-Lipschitz maps $f:\Sp^1\to \Sp^1$).
\end{remark}

What Theorem \ref{thm-main-1} theorem shows is that the set of $K$-qc maps of $\Sp^{2}$ that admit harmonic quasi-isometric extensions is  closed with respect to pointwise convergence. That is,  if $f_N$ is a sequence of $K$-quasiconformal maps that admit harmonic and quasi-isometric extensions, and if $f_N$ pointwise converges to (a necessarily) $K$-quasiconformal map $f$, then $f$ also admits a harmonic quasi-isometric extension (this is a standard argument and it  follows from Cheng's lemma \cite{cheng} and the Azrela-Ascoli theorem).

The following theorem is   a corollary of Theorem \ref{thm-main-1} and the result of Li and Tam \cite{l-t-1} that every diffeomorphism of $\Sp^{2}$ admits a harmonic quasi-isometric extension. It is a  fact that  every quasiconformal map of the 2-sphere is a limit of uniformly quasiconformal diffeomorphisms.

\begin{theorem}\label{thm-main-2} Every quasiconformal map of $\Sp^2$ admits a harmonic quasi-isometric extension. 
\end{theorem}

As we said, every quasiconformal map of the 2-sphere is a limit of uniformly quasiconformal diffeomorphisms. Whether such approximation result holds in higher dimensions is a hard open problem. In dimension 4, it is related to the question of whether a 4-dimensional quasiconformal manifold carries a differentiable structure. Thus, extending Theorem \ref{thm-main-1} to higher dimensions would not directly imply the generalization of Theorem \ref{thm-main-2}  to higher dimensions (below we briefly discuss a somewhat different approach, based on the same circle of ideas, that may lead to proving Theorem \ref{thm-main-2} to all dimensions $n \ge 3$).

\subsection{An interesting sequence of harmonic maps on the unit disc} As promised above, we construct a sequence of diffeomorhisms $f_n:\Sp^1 \to \Sp^1$, that converge to the identity in the $C^0$ sense, but such that the corresponding harmonic extensions degenerate on  compact sets in $\Ha$ and we can not extract any sort of limit. 

We construct harmonic maps by prescribing their Hopf's differentials (see \cite{wan}, \cite{t-w-0}). Let $\D=\{z \in \C: \, |z|<1\}$ denote the unit disc and let $\rho(z)$ denote the density of the hyperbolic metric on $\D$. 
The Hopf differential of a harmonic map $f:\D \to \D$ is given by 
$$
\text{Hopf}(f)=\rho^{2}(f)(\partial{f}) (\overline{\partial}f) \, dz^{2}. 
$$
A harmonic map on $\D$ is uniquely determined (up to the post-composition by a M\"obius transformation) by its Hopf differential. If the Hopf differential is smooth up to the boundary then the corresponding harmonic diffeomorphism is  smooth up to the boundary  (see \cite{l-t-1}).

Set 
$$
\varphi_n(z)=2^{n-2} z^{n-2}, \,\, z \in \D,
$$
for $n \in \N$, and denote by $f_n:\D \to \D$ the harmonic map whose Hopf differential is equal to $\varphi_n\,dz^2$ and normalized so that $f_n(0)=0$ and $f_n(1)=1$ (here $0 \in \D$ and $1 \in S^1=\partial \D$).

Define the rotation $R_n(z)$ by
$$
R_n(z)=e^{\frac{2\pi i}{n}} z, \,\, z \in \D,
$$
and observe that
$$
\big( \varphi_n \circ R_n \big)(R'_n)^{2}=\varphi_n.
$$
This implies that the maps $f_n$ and $f_n \circ R_n$ have the same Hopf differential which yields that $f_n \circ R_n=A_n \circ f_n$, for some M\"obius transformation of $\D$.
Since both $f_n$ and $f_n \circ R_n$ fix the origin $0$, we conclude that $A_n$ is a rotation (fixing the origin). Iterating the identity $f_n \circ R_n=A_n \circ f_n$ we obtain the identities
$$
f_n \circ R^{k}_n=A^{k}_n \circ f_n,
$$
for any $k \in \Z$. Letting $k=n$ shows that $A^{n}_n$ is the identity and we conclude that $A_n=R_n$.
 
It follows from the identities 
$$
f_n \circ R_n=R_n \circ f_n,
$$
and $f_n(1)=1$ that $f_n$ converges to the identity map on $S^1$ when $n \to \infty$. On the other hand, the sequence $\varphi_n\, dz^2$ of the corresponding Hopf differentials does not have a limit (it blows up on the annulus $1/2<|z|<1$) and therefore the sequence $f_n$ does not have a limit.

\subsection{The main ideas} All notation we introduce here will be defined in more details later. 
We identify the hyperbolic space $\Ho$ with the unit ball model $\Bnn \subset \R^3$. By $\QC(\Sp^2)$ we denote  normalized (fixing the same three distinct points) quasiconformal maps of $\Sp^2$ and by $\QI(\Bnn)\equiv \QI(\Ho)$ the space of quasi-isometries of $\Ho$.

Each point $x\in \Bnn$ can be written in the polar coordinates $x=\rho \zeta$, where $\zeta \in \Sp^2$ and $\rho=|x|$. By $\sigma$ we denote the probability Lebesgue measure on $\Sp^{2}$ and by $\lambda$  the measure on $\Bnn$ such that $d\lambda$ is the volume element with respect to the hyperbolic metric on $\Bnn$. 

For a continuous function $F$ on $\Bnn$, we let 
$$
\Avr_{F}(\rho)=\int\limits_{\Sp^{2} } F(\rho \zeta) \, d\sigma(\zeta).
$$
For a $C^2$ function $F$ on $\Bnn$, the following corollary of the Green's Identity holds true
\begin{equation}\label{out-1}
F(0)+ \int\limits_{\Bnn}  \green_r(x) \lap F(x) \, d\lambda(x)=\Avr_F(r),\,\, 0\le r<1,
\end{equation}
where $\lap F$ is the Laplacian of $F$ computed with respect to the hyperbolic metric, and $\green_r$ is the hyperbolic Green function for the ball $r\Bnn$ (the Green function is defined on the product $r\Bnn \times r\Bnn$ but here we abuse notation slightly but letting $\green_r(x)=\green_r(x,0)$).

The proof of Theorem \ref{thm-main-1} is based on the following three main ingredients:

\begin{itemize}

\item For every $f \in \QC(\Sp^2)$, there exists a ``good" extension $\good:\QC(\Sp^2) \to \big(C^2(\Bnn) \cap \QI(\Bnn) \big)$. Among other things, the ``good" extension  $\good(f)$ is ``close" to being harmonic at a random point of $\Bnn$ (by ``close" to being harmonic we mean that $\good(f)$ nearly satisfies the equation that defines harmonic maps, see below for the precise statement). We construct such an extension explicitly in the last section of this paper (this extension is very similar to the 
generalized Ahlfors-Beurling extension).

\item The inequality, computed by Schoen and Yau in \cite{s-y-1} (see the paper \cite{j-k} by Jager and Kaul for very similar computations) that gives a lower bound for the Laplacian $\lap \dista^2(f)$, where  
$$
\dista(f)(x)=d_{\Ho}(\good(f)(x),H(f)(x)), 
$$
and $H(f)$ is the harmonic quasi-isometric extension of $f$ (providing it exists). 

\item Although we do not construct $\good$ to be conformally natural, after replacing   $f$ by $I\circ f \circ  J$ ($I,J \in \Isom(\Ho)$) if necessary, we may assume that $\dista(f)(0) \ge ||\dista(f)||-D$, for some constant $D=D(K)$ that depends only on $K$ (here $||\dista(f)||$ is the supremum of $\dista(f)(x)$, for $x \in \Bnn$).  Applying (\ref{out-1})  to $\dista^2(f)$, we get the main estimate
\begin{align}\label{out-2}
||\dista(f)||^{2}-D' ||\dista(f)||-D'' + &\int\limits_{\Bnn} \green_r(x) \lap \dista^2(f)(x)\, d\lambda(x) \le \\
\nonumber &\le \Avr_{\dista^{2}(f)}(r) \le ||\dista(f)||^{2}.
\end{align}
\end{itemize}
To prove Theorem \ref{thm-main-1}, it suffices to prove
\begin{equation}\label{out-3}
||\dista(f)|| \le D_1,
\end{equation}
for some constant $D_1=D_1(K)$ and a $K$-quasiconformal (or just $K$-qc) map $f$.
\vskip .3cm
The crucial estimate (\ref{out-3}) follows from (\ref{out-2}) and the following inequality: There exists $0<r_0=r_0(K,M)<1$ and $D'''=D'''(K,M)$, such that 
$$
\int\limits_{\Bnn} \green_{r_{1}}(x) \lap \dista^2(f)(x)\, d\lambda(x)  \ge (D'+1)||\dista(f)||-D''',
$$
for some $0<r_1\le r_0$, and every $K$-qc map $f$.

The previous inequality will be proved using the properties of the good extension and the lower bounds for  $\lap \dista^2(f)(x)$.  We provide a more detailed outline in the next subsection.

We conclude this part of the introduction with the following three remarks.

\begin{remark} Even if a qc map $f \in \QC(\Sp^2)$ does not have a harmonic quasi-isometric extension, we can always find the map $H_r(f): r\Bnn \to \Bnn$, which is harmonic and it agrees with $\good(f)$ on the sphere $r\Sp^2$. One may ask if  this method can be somehow applied to bound above the distance between $\good(f)$ and $H_r(f)$.  All our estimates are taking place on a  ball in $\Bnn$, that is centered at some point where $\dista(f)$ is close to its supremum, and has a sufficiently large radius (this is at the heart of our argument). But if we consider the distance function $d_{\Ho}(\good(f)(x),H_r(f)(x))$, for $x$ in $r\Bnn$, the point where the distance is close to its maximum may be near the boundary of $r\Bnn$, and thus $r\Bnn$ may not contain the sufficiently large ball we need, so our method can not directly be applied in this situation. 
\end{remark} 
\vskip .3cm
\begin{remark} The existence of a ``good" extension was  used to prove that a diffeomorphism of $\Sp^{n-1}$ admits a harmonic quasi-isometric extension to $\Hnn$.  When $f$ is sufficiently smooth on $\Sp^{n-1}$, then the good extension has the property that the norm $|\tau(\good(f))|$ of its tension field is in $L^2(\Hnn)$. In general, when $f$ is only quasiconformal,  the ``good" extension we construct does not have  strong enough properties and we can not apply the standard heat flow method developed by Li-Tam. However, it is very reasonable to expect that the heat flow whose initial map is a  ``good" extension actually converges, and this author aims to explore this further. If so, this would give a proof of Theorem \ref{thm-main-2} in all dimensions $n\ge 3$.
\end{remark}

\begin{remark}
The reason we can construct the good extension $\good(f)$ is because $f$ is differentiable almost everywhere on $\Sp^2$, which is a manifestation of quasiconformal rigid. This is not true for maps of the unit circle and it is not clear how to construct such an extension in this case (unless perhaps if the boundary map is bi-Lipschitz to begin with). 
\end{remark}

\subsection{A more detailed outline}  Given a map $F \in C^{2}(\Bnn)$, by $\ener(F)(x)$ we denote the energy of $F$ and by $\tau(F)(x)$ the tension field.  Then $F$ 
is said to be harmonic if 
$$
\tau(F)(x)=0,\,\,\, x\in \Ho.
$$

The extension $\good$ is continuous: if $f_n$ is a sequence of $K$-qc maps that pointwise converges to some $f \in \QC(\Sp^{2})$, then $\good(f_n) \to \good(f)$ in $C^{2}(\Bnn)$ (that is, the first and second derivatives of $\good(f_n)$ converge  respectively to the first and second derivatives of  $\good(f)$, uniformly on compact sets in $\Bnn$).  Also, the inequality $||\tau(\good(f)||\le T=T(K)$, holds for every $K$-qc  $f \in \QC(\Sp^2)$.

We explain now what it means that $\good(f)$ is ``close" to being harmonic at a random point. For $\epsilon>0$, $K_1\ge 1$, and $f \in \QC(\Sp^2)$, we define 
$X_f(K_1,\epsilon) \subset \Bnn$ by letting $x \in X_f(K_1,\epsilon)$ if  $|\tau(\good(f))(x)| < \epsilon$, $\dist(\good(f))(x)< K_1$, and $\ener(\good(f))(x)>1$ (here $\dist$ denotes the quasiconformal distortion). 
If we keep $K_1$  fixed and let $\epsilon$ be small, then at such points the map $\good(f)$ is close to being harmonic (and since $\dist(\good(f))(x)< K_1$, the map $\good(f)$ is locally $K_1$-quasiconformal on $X_f$).

\begin{remark} We point out that harmonic maps are  not necessarily everywhere   ``close" to being harmonic in our sense. Although the tension field of a harmonic  map is zero, they do not have to be locally quasiconformal nor their energy has to be greater that 1. 
\end{remark}

Theorem \ref{thm-extension} below  says that for every $\epsilon>0$ and a $K$-qc map $f$, we have 
\begin{equation}\label{out-4}
\lim_{\rho \to 1} \sigma\big( \{\zeta \in \Sp^{2}:\, \rho\zeta \in X_f(2K,\epsilon)\} \big)   \to 1.
\end{equation}
\vskip .3cm

The convergence in (\ref{out-4}) may not be uniform over all $K$-qc mappings, but using the continuity of $\good$ we prove a uniformity statement  (see Lemma \ref{lemma-pomocna} below) that is sufficient for us.

The way we use that $\good(f)$ is close to being harmonic on $X_f(\epsilon)=X_f(2K,\epsilon)$ is as follows. One easily computes $\epsilon_0=\epsilon_0(K)>0$, such that for $X_f=X_f(\epsilon_0)$,  from the lower bound of $\lap \dista^2(f)$ we get

\begin{equation}\label{out-5}
\lap \dista^2(f)(x) \ge c||\dista(f)|| >0,\,\,\ x \in X_f\bigcap \{x \in \Bnn:\, \dista(f)(x)\ge \frac{1}{2} \}, 
\end{equation}
and
\begin{equation}\label{out-6}
\lap \dista^2(f)(x) \ge -T||\dista(f)||,\,\,\, \text{for every}\,\, x \in \Bnn,
\end{equation}
for some constant $c=c(K)>0$, and where $T=T(K)$ was defined above. 

From (\ref{out-6}) and (\ref{out-2}), we bound above the measure of the set 
$$
\sigma\left( \{\zeta \in \Sp^2:\, \dista(f)(r\zeta)  \le \frac{||\dista(f)||}{2} \} \right) \le \frac{\varphi_K(r)}{||\dista(f)||}, 
$$
where $\varphi_K(r)$ is a fixed function of $r \in [0,1)$, for a given $K$ (see Lemma \ref{lemma-measure} below). Using this bound, the estimate (\ref{out-5}), and the uniformity version of (\ref{out-4}), we show (see Lemma \ref{lemma-estimate}) that for every $M>0$, there exists $0<r_0=r_0(K,M)<1$, such that
\begin{equation}\label{out-7}
\int\limits_{\Bnn} \green_{r_{1}}(x) \lap \dista^2(f)(x)\, d\lambda(x) \ge M||\dista(f)||-\psi_K(r_0),
\end{equation}
for some $0<r_1\le r_0$, where $\psi_K(r)$ is a fixed function for a given $K$.  

Let $M=M(K)=D'+1$.  Then by (\ref{out-7}) we have
$$
\int\limits_{\Bnn} \green_{r_{1}}(x) \lap \dista^2(f)(x)\, d\lambda(x) \ge M||\dista(f)||-\psi_K(r_0)=(D'+1)||\dista(f)||-\psi_K(r_0).
$$
Replacing this in (\ref{out-2}), we get
$$
||\dista(f)|| \le \psi_K(r_0)+D''=D_1=D_1(K),
$$
thus proving (\ref{out-3})

\subsection{Organization of the paper}

In Section 2 we recall basic definitions and needed formulas in $\Bnn$. Section 3 is devoted to the definition of the extension $\good$  and its properties. In Section 4 we prove Theorem \ref{thm-main-1} assuming the existence of $\good$. Sections 2,3, and 4 closely  follow the above outline (these sections are really the expanded versions of outline).
  
In the last section we explicitly construct $\good$ by working in the upper-half space model of $\Ho$. The extension $\good$ is a version of the well known Ahlfors-Beurling construction \cite{a-b}.

\subsection{Acknowledgement} I am grateful to the referee for his/her comments and suggestions. Most of this project was carried out while the author was visiting University of Minnesota in Minneapolis as the \textsl{Ordway Visiting Professor}. I wish to thank them for their hospitality.

\section{Preliminaries}

\subsection{Quasi-isometries and quasiconformal maps}  For the relevant background on quasiconformal maps and quasi-isometries of hyperbolic spaces see \cite{pansu}, \cite{l-w}, \cite{t-w}.
Let $F:X \to Y$ be a map between two metric spaces $(X,d_X)$ and $(Y,d_Y)$.  We say that $F$ a $(L,A)$-quasi-isometry 
if there are constants $L>0$ and $A \ge 0$, such that
$$
\frac{1}{L}d_Y(F(x),F(y))-A \le d_X(x,y) \le Ld_Y(F(x),F(y))+A,
$$
for every $x,y \in X$ (some authors call this a rough isometry but we stick to the name commonly used in hyperbolic geometry).  
By $\QI(X,Y)$ we denote the set of all  quasi-isometries  from $X$ to $Y$ and if $X=Y$ then $\QI(X,X)\equiv \QI(X)$.

We define the distortion function $\dist(F)(x)$ by
$$
\dist(F)(x)=\limsup\limits_{t \to 0}  \,   \frac{ \max\limits_{d_{X}(x,y)=t} d_Y(F(x),F(y)) }  {\min\limits_{d_{X}(x,y)=t} d_Y(F(x),F(y)) }.
$$
If $\dist(F)(x) \le K$ on some  set $U \subset X$, we say that $F$ is locally $K$-quasiconformal on $U$. If $F$ is a global homeomorphism and $\dist(F)(x) \le K$ for every $x \in X$, we say that $F$ is $K$-quasiconformal.  

We will be consider quasiconformal maps of $\Sp^2$ and locally quasiconformal maps between  open subsets of $\Ho$.

Each quasi-isometry $F:\Ho \to \Ho$ extends continuously to  $\overline{\Ho}$  and  the restriction of $F$ to $\Sp^{2}$ is  a quasiconformal map. 
Moreover, if $F$ and $G$ are two quasi-isometries of $\Ho$  that have the same boundary values, then the distance $d_{\Ho}(F(x),G(x))$ is bounded on $\Ho$.

\begin{definition} Once and for all we fix three distinct points on $\Sp^2$. A quasiconformal map of $\Sp^2$ is normalized if it fixes these three  points. The set of all normalized quasiconformal maps from $\Sp^2$ to itself is denoted by $\QC(\Sp^2)$.
\end{definition}

\subsection{Tension, energy and the distance} For background on harmonic maps see \cite{s-y}. 
Given Riemannian manifolds $(M,g)$, $(N,h)$ and a $C^2$ map $F:M \to N$, the energy density of $F$ at a point $x \in M$ is defined as
$$
\ener(F) = \frac{1}{2} |dF|^2
$$
where the $|dF|^2$ is the squared norm of the differential of $F$, with respect to the induced metric on the bundle $T^* M \times F^{-1} T N$. It can also  be written as

$$
\ener(F) = \frac{1}{2} \text{trace}_g F^* h.
$$

In local coordinates this  reads as
$$
\ener(F) = \frac{1}{2} g^{ij} h_{\alpha\beta}\frac{\partial{F}^\alpha}{\partial x^i}\frac{\partial{F}^\beta}{\partial x^j}.
$$

The tension field of $F$ is given by 
$$
\tau(F)= \text{trace}_g \nabla dF,
$$
where $\nabla$ is the connection on the vector bundle $T^* M \times F^{-1} T N$  induced by the Levi-Civita connections on M and N. 
\vskip .3cm

We let $\Bnn$ denote the unit ball in $\R^3$. By  $C^{k}(\Bnn,\R^3)$ we denote the topological space of $k$ differentiable mappings from $\Bnn$ into $\R^3$ equipped with the standard $C^{k}$ topology.
By $C^k(\Bnn)\equiv C^k(\Ho)$ we denote the closed subspace of   $C^{k}(\Bnn,\R^3)$ that contains those maps that map $\Bnn$ into itself (here we identify the hyperbolic space $\Ho$ with its unit ball model $\Bnn$) .

As in the above definition, given a map $F \in C^{2}(\Ho)$ by $\ener(F)(x)$ we denote the energy of $F$ at $x \in \Ho$. By $\tau(F)(x)$ we denote its tension field and by $|\tau(F)(x)|$ the norm of the tension field at $x$. Recall that $F$ is a harmonic map if $\tau(F) \equiv 0$. We let $||\tau(F)||=\sup_{x\in \Ho} |\tau(F)(x)|$  and $||\ener(F)||=\sup_{x\in \Ho} |\ener(F)(x)|$ (the tension field and the energy of $F$ are computed with respect to the hyperbolic metric).

\begin{definition} Given two mappings $F, G, \in \C^2(\Ho)$, we define the function $\dista:\Ho \to [0,\infty)$ by 
$$
\dista(x)=d_{\Ho} (F(x),G(x)), \, x \in \Ho.
$$
\end{definition}

Function $\dista^2$ is $C^2$, and Schoen and Yau  computed its Laplacian (see page 368 in \cite{s-y-1}, and also the papers \cite{j-k} and \cite{h-w}  by  Jager-Kaul and Hardt-Wolf for similar computations).  The following  well known inequality is a corollary of the formula from \cite{s-y-1}. It  was stated in many papers (see  \cite{d-w}, \cite{h-w}, \cite{l-t-2}, \cite{yang}, \cite{t-w} ).

Let $u_1,u_2,u_3$ be an orthonormal frame at $x \in \Ho$ and let $v_1,v_2,v_3$ and $w_1,w_2,w_3$ denote the orthonormal frames at the points $F(x)$ and $G(x)$ respectively such that the vectors $v_3$ and $w_3$ are tangent to the geodesic segment  connecting $F(x)$ and $G(x)$ and point away from each other.  Let 
$$
F_*(u_i)=\sum_{j=1}^{3} \alpha_{i}^{j}(x) v_{j},
$$
and
$$
G_*(u_i)=\sum_{j=1}^{3} \beta_{i}^{j}(x) v_{j}.
$$

Denote by $\lap\dista^2(x)$ the Laplacian of $\dista^2(x)$ computed with respect to the hyperbolic metric. Then the inequality 

\begin{align}\label{distance}
\lap \dista^2(x) &\ge -2\dista(x)\big( |\tau(F)(x)|+|\tau(G)(x)| \big) +\\  
\nonumber & + 2\dista(x) \sum_{i=1}^{3}\sum_{j=1}^{2}  \big( (\alpha_{i}^{j}(x))^{2}+(\beta_{i}^{j}(x))^{2} \big) \tanh \frac{\dista(x)}{2},
\end{align}
holds for every  $x \in \Bnn$.

The following lemma is a corollary of (\ref{distance}) and it summarizes exactly how the inequality (\ref{distance})  will be applied in the proof of our main theorem.

\begin{lemma}\label{lemma-distance} If $F,G \in C^{2}(\Ho)$, then 
\begin{equation}\label{distance-1}
\lap \dista^2(x) \ge -2\dista(x)\big( ||\tau(F)||+||\tau(G)|| \big), \, x \in \Ho.
\end{equation}
Moreover, let $K_1 \ge 1$ and let  $U \subset \Ho$ be the set where the inequality $\dist(F)(x)\le K_1$ holds. There exists a constant $q=q(K_1)>0$ such that 
\begin{align}\label{distance-2}
\lap \dista^2(x) \ge & -2\dista(x)\big(|\tau(F)(x)|+|\tau(G)(x)|\big) + \\
\nonumber &+ 2q\dista(x) \ener(F)(x) \tanh \frac{\dista(x)}{2}, \, x \in U.
\end{align}

\end{lemma} 

\begin{proof}
The inequality (\ref{distance-1}) is an immediate corollary of (\ref{distance}). To prove the inequality (\ref{distance-2}) we need to show 
$$
\sum_{i=1}^{3}\sum_{j=1}^{2}  \big( (\alpha_{i}^{j}(x))^{2}+(\beta_{i}^{j}(x))^{2} \big) \ge q \ener(F)(x) , \, x \in U.
$$
This was proved by Tam and Wan (see page 12 in \cite{t-w}) and we outline their argument.   By elementary linear algebra there exists a constant $K_2=K_2(K_1)$ such that if $\dist(F)(x)\le K_1$ then
$$
\ener(F)\le K_2 J^{\frac{2}{3}}(F)(x),
$$
where $J(F)$ is  the Jacobian of $F$. Then (again as an exercise in linear algebra) they show that there exists a universal constant $C$  such that 
$$
J^{\frac{2}{3}}(F)(x)\le C \sum_{i=1}^{3}\sum_{j=1}^{2} \big( \alpha_{i}^{j}(x) \big)^{2}.
$$
We let 
$$
q=\frac{1}{C K_2},
$$
and observe that $q$ depends only on $K_1$. Thus, we have the inequality 
$$
q\ener(F)(x)\le  \sum_{i=1}^{3}\sum_{j=1}^{2} \big( \alpha_{i}^{j}(x) \big)^{2},\,\, x \in U,
$$
which together with (\ref{distance}) implies that the inequality (\ref{distance-2}) holds for $x \in U$.
 
\end{proof}

\subsection{The unit ball model of the hyperbolic space} As above, by $\Bnn$ we denote the unit ball in $\R^3$ and by $\Sp^{2}=\pt{\Bnn}$ the unit sphere. Every $x \in \Bnn$ can be written as $x=\rho\zeta$, where $\rho$ is the absolute value of $x$ and $\zeta \in \Sp^{2}$ the corresponding point ($\zeta$ is uniquely determined by $x$ when $|x|\ne 0$). 
All formulas we state below are classical and can be found  in  chapters three and four of \cite{stoll}, chapter five of  \cite{ahlfors}, and chapter four in \cite{rudin}. Our exposition follows the survey article by Stoll \cite{stoll} which in turn very closely follows the exposition in
\cite{pavlovic}.

We let $\sigma$ denote the Lebesgue measure on $\Sp^{2}$ normalized to be a probability measure, that is $\sigma(\Sp^{2})=1$, and  
we let $\mu$ denote the Lebesgue measure on $\R^3$, normalized to be a probability measure on $\Bnn$, that is $\mu(\Bnn)=1$.

By $\lambda$ we denote the measure on $\Bnn$ such that 
$$
d\lambda(x)=\frac{d\mu(x)}{(1-|x|^2)^3}.
$$
Then $d\lambda$ is a constant multiple of the volume element with respect to the hyperbolic metric on $\Bnn$.

In polar coordinates $x=\rho\zeta$, $0\le \rho<1$ and $\zeta \in \Sp^{2}$, we have (taking into account the normalizations we imposed on $\sigma$ and $\mu$)
$$
d\mu(x)=3 \rho^{2} \, d\rho d\sigma(\zeta) 
$$
and thus
$$
d\lambda(x)=\frac{3\rho^{2}\, d\rho d\sigma(\zeta)}{(1-\rho^2)^3}.
$$

\begin{definition}\label{def-average} For a continuous function $F:\Bnn \to \R$, we define the spherical average 
$$
\Avr_{F}(\rho)=\int\limits_{\Sp^{2} } F(\rho \zeta) \, d\sigma(\zeta), 
$$
for every $0\le \rho <1$.

\end{definition}

We say that a function on $\Bnn$ is radial if it is constant on each sphere $\zeta\Sp^{2} \subset \Bnn$ (that is, the value of the function at $x \in \Bnn$ depends only on $|x|$).

\begin{lemma}\label{lemma-average} Let $F$ and $\Phi$ be continuous functions on $\Bnn$ and assume that $\Phi$ is a radial function. Then the identity
$$
\int\limits_{\Bnn} \Phi(x)F(x) \, d\lambda(x)=  \int\limits_{0}^{1} \frac{3\rho^{2}\Phi(\rho) \Avr_F(\rho)}{(1-\rho^{2})^3} \, d\rho  =\int\limits_{\Bnn} \Phi(x) \Avr_F (x)  \, d\lambda(x),
$$
holds.
\end{lemma}
\begin{proof} Passing to polar coordinates we get 

\begin{align*}
\int\limits_{\Bnn}  \Phi(x) F(x) \, d\lambda(x) & = \int\limits_{0}^{1}  \int\limits_{\Sp^{2}}  \frac{3\rho^{2}\Phi(\rho)F(\rho \zeta)}{(1-\rho^{2})^3} \, d\rho d\sigma(\zeta) \\
&=\int\limits_{0}^{1} \frac{3\rho^{2}\Phi(\rho) \Avr_F(\rho)}{(1-\rho^{2})^3} \, d\rho \\
&= \int\limits_{\Bnn} \Phi(x) \Avr_F(x)  \, d\lambda(x).
\end{align*}

\end{proof}

See Definition 3.3 in \cite{stoll} and formula (3) in \cite{ahlfors} for the definition of the Green function on a hyperbolic space.
By $\green_r$ denote the Green function (with respect to the hyperbolic Laplacian $\lap$)  for the ball $r\Bnn$.  Then $\green_r$ is a radial function and is given by
$$
\green_r(x)=\frac{1}{3}\int\limits_{|x|}^{r} \frac{(1-s^{2})}{s^{2}}\, ds,
$$
for $|x| \le r$ and $\green_r(x)=0$ when $r<|x|<1$. The function $\green_r$ is continuous on $\Bnn$.

The identity in the next lemma is obtained as the limit when $\epsilon \to 0$ of  the second Green formula   applied to the functions $F$ and $\green_r$ on the annulus $0<\epsilon\le |x|\le r$  
(see Theorem 4.2 in \cite{stoll}).
\begin{lemma}\label{lemma-green}  Let $F:\Bnn \to \R$ be a  $C^2$ function. Then
\begin{equation}\label{eq-green}
F(0)+ \int\limits_{\Bnn}  \green_r(x) \lap F(x) \, d\lambda(x)=\Avr_F(r)
\end{equation}
for each $0\le r<1$ (we recall that $\lap F(x)$ is computed with respect to the hyperbolic metric on $\Ho$).
\end{lemma}

By $\green \equiv \green_1$ we denote the Green function for $\Bnn$. We will need the following standard estimate for $\green$
\begin{equation}\label{eq-green-est}
\const_1 \frac{(1-\rho^{2})^{2}}{\rho} \le \green(\rho) \le \const_2 \frac{(1-\rho^{2})^{2}}{\rho},
\end{equation}
where $\const_1$ and $\const_2$ are universal constants and $\green(x)=\green(|x|)=\green(\rho)$. 

We finish this subsection by observing that $0\le \green_r (x) <\green(x)$, $x \in \Bnn$, and that $\green_r$ converges to $\green$, when $r \to 1$,  uniformly on compact subsets of $\Bnn$.

\section{Good and admissible extensions of quasiconformal maps} In this section we state the theorem that a good  ($\good$) extension exists. In the last subsection we prove an important property of $\good$.

\subsection{Admissible extensions} We are seeking harmonic quasi-isometric extensions of quasiconformal maps of the sphere $\Sp^{2}$. On the other hand, it is well known that each quasiconformal map 
$f \in \QC(\Sp^{2})$ can be extended to a smooth quasi-isometry. The Douady-Earle \cite{d-e} (or the barycentric) extension  is an example. In fact, the Douady-Earle extension is an example of an \textsl{admissible} extension that we now define.

\begin{definition}\label{def-admissible} We say that a mapping $\Ext: \QC(\Sp^{2}) \to (C^{2}(\Ho) \cap \QI(\Ho) )$ is an admissible extension if it has the following properties:

\begin{enumerate}
\item \textbf{Uniform quasi-isometry}: There are constants $L=L(K) \ge 1$ and $A=A(K) \ge 0$  such that if $f$ is $K$-quasiconformal, then $\Ext(f)$ is 
a $(L,A)$-quasi-isometry.

\item \textbf{Uniformly bounded tension}: There exists a constant $T=T(K)>0$  such that 
$$
||\tau(\Ext(f))||\le T.
$$

\item \textbf{Continuity}: If $f_n$ is a sequence of $K$-quasiconformal maps that pointwise converges to some $f \in \QC(\Sp^{2})$, then $\Ext(f_n) \to \Ext(f)$ in $C^{2}(\Ho)$ (that is, the first and second derivatives of $\Ext(f_n)$ converge 
respectively to the first and second derivatives of  $\Ext(f)$, uniformly on compact sets in $\Ho$).

\end{enumerate}

\end{definition}

Beside being an admissible extension, the Douady-Earle extension has further properties, like being conformally natural, real analytic, etc., but  this definition of admissible extension is sufficient for our purposes.

Recall that $\QC(\Sp^2)$ is the space of normalized quasiconformal mappings of $\Sp^2$.  We do not assume that  an admissible extension is conformally natural, that is if $f \in \QC(\Sp^2)$, and $I,J \in \Isom(\Ho)$ are two M\"obius maps such that $I \circ f \circ J$ is again normalized, then we do not require that $I\circ \Ext(f)\circ J=\Ext(I\circ f \circ J)$ . 
The following lemma offers a suitable replacement.

\begin{lemma}\label{lemma-distance-bound} Let $\Ext$ denote an admissible extension. Then for each $K \ge 1$ there exists a constant $D=D(K,\Ext)$  such that for every $K$-quasiconformal map $f \in \QC(\Sp^{2})$ and any two isometries $I,J \in \Isom(\Ho)$ such that $I \circ f \circ J$ is normalized,  we have
$$
\sup_{x \in \Ho} d_{\Ho} \big( (I \circ \Ext(f)\circ J)(x),\Ext(I\circ f \circ J)(x) \big) \le D.
$$
\end{lemma}

\begin{proof} The maps  $I \circ \Ext(f)\circ J$ and $ \Ext(I\circ f \circ J)$ are both  $(L,A)$-quasi-isometries, where $L$ and $A$ depend only on $K$. Since these two maps agree on $\Sp^{2}$, it follows that their distance is bounded by a constant that only depends on $K$ (see \cite{pansu} or \cite{l-w}) .

\end{proof}

\subsection{The $\good$-extension} 

Given an admissible extension $\Ext(f)$, $f \in \QC(\Sp^{2})$,  we want to identify  points in $\Bnn$ at which this extension is close to being harmonic and locally quasiconformal. We make this precise as follows.

\begin{definition} Let $K_1, \epsilon>0$ and  let $f \in \QC(\Sp^{2})$ denote a $K$-qc map. Define the (open) set  $X_f^{\Ext}(K_1,\epsilon) \subset \Ho$ by letting $x \in X_f^{\Ext}(K_1,\epsilon)$ if the following conditions hold:

\begin{enumerate}
\item $|\tau(\good(f))(x)| < \epsilon$,
\item $\dist(\good(f))(x)< K_1$, and
\item  $\ener(\good(f))(x)>1$.
\end{enumerate}

\end{definition}

When $K_1$ is being kept fixed and we let $\epsilon$ be small, then $\Ext(f)$ is close to being harmonic on the open set  $X_f^{\Ext}(K_1,\epsilon)$ (by definition, the map $\Ext(f)$ is locally $K_1$-quasiconformal on  $X_f^{\Ext}(K_1,\epsilon)$).

The following theorem can be interpreted as saying  that there exists an admissible extension that is ``close to being harmonic" at a random point in $\Ho$. 
We now pass onto the unit ball model $\Bnn$ of the hyperbolic space $\Ho$.

\begin{theorem}\label{thm-extension} There exists an admissible extension $\good: \QC(\Sp^{2}) \to (C^{2}(\Ho) \cap \QI(\Ho) )$ such that for  every $K\ge 1$, $\epsilon>0$, 
and  a $K$-qc map $f \in \QC(\Sp^{2})$, we have

\begin{equation}\label{convergence}
\lim_{\rho \to 1} \sigma\big( \{\zeta \in \Sp^{2}:\, \rho\zeta \in X^{\good}_f(2K,\epsilon)\} \big)   \to 1.
\end{equation}

\end{theorem}

\begin{remark}
We will prove Theorem \ref{thm-extension} in the last section of this paper. Until then, we fix once and for all such an extension $\good$, 
and for each $\epsilon>0$ and a $K$-qc map $f \in \QC(\Sp^2)$, we let
\begin{equation}\label{eq-X}
X_f(\epsilon)=X_f^{\good}(2K,\epsilon).
\end{equation}
We also fix the constant $T=T(K)$ from the definition of an admissible extension, that is $||\tau(\good(f))||\le T$, for every $K$-quasiconformal map $f$.
\end{remark}

\subsection{Uniformity properties of $\good$} The convergence in (\ref{convergence}) may not be uniform over all $K$-quasiconformal mappings $f \in \QC(\Sp^{2})$, but since $\good$ is continuous (see the property  $(3)$ in Definition \ref{def-admissible}) we can still prove a similar (but weaker) statement that is quite sufficient for our purposes.  The following lemma is a corollary of (\ref{convergence}) and the continuity of $\good$, and it will be applied in the proof of Theorem \ref{thm-main-1} below.

\begin{lemma}\label{lemma-pomocna}  Let  $C_1,C_2,\epsilon_0>0$ and $K \ge 1$.
For each $K$-quasiconformal  $f \in \QC(\Sp^{2})$ we let 
$$
\Phi(f,C_1,C_2)=\Phi(f)(x)=
\begin{cases}
C_2,\,\, \text{if}\,\,x \in   X_f(\epsilon_0)    ,\\
-C_1,\,\,  \text{otherwise}.
\end{cases}
$$
Then for every $M>0$, there exists $0<r_0=r_0(M,K,C_1,C_2)<1$ such that 
\begin{equation}\label{eq-formula}
\int\limits_{\Bnn} \green_{r_{1}}(x)\Phi(f)(x) \, d\lambda(x)>M,
\end{equation}
for some $0<r_1\le r_0$.
\end{lemma}

\begin{proof}  We first show that 
\begin{equation}\label{eq-formula-3}
\lim\limits_{r \to 1} \int\limits_{\Bnn} \green_{r} (x)\Phi(f)(x) \, d\lambda(x)=\infty,
\end{equation}
for each $f \in  \QC(\Sp^{2})$.

\vskip .3cm
By Lemma \ref{lemma-average}, we have 
\begin{equation}\label{eq-1}
\int\limits_{\Bnn} \green_r(x)\Phi(f)(x) \, d\lambda(x)=\int\limits_{0}^{r} \frac{3\rho^{2}\green_r(\rho)\Avr_{\Phi(f)}(\rho)}{(1-\rho^{2})^3} \, d\rho.
\end{equation}

It follows from the definition of $\Phi(f)$ that
\vskip .3cm
$$
\Avr_{\Phi(f)}(\rho)= C_2 \sigma\big( \{\zeta \in \Sp^{2}:\, \rho \zeta \in X_f(\epsilon_0) \} \big)-C_1  \big(1-\sigma\big( \{\zeta \in \Sp^{2}:\, \rho \zeta \in X_f(\epsilon_0) \} \big).
$$
\vskip .3cm
\noindent
This inequality  together with (\ref{convergence}), implies that there exists $0<r_0=r_0(f,C_1,C_2)<1$, such that 
$$
\Avr_{\Phi(f)}(\rho)>\frac{C_2}{2},\,\, \text{for every} \,\, r_0 \le \rho<1.
$$

Plugging this into (\ref{eq-1}) shows that for every $r_0<r_1 \le r$, we have the inequality 
\begin{equation}\label{eq-2}
\int\limits_{\Bnn} \green_r(x)\Phi(f)(x) \, d\lambda(x) \ge \frac{C_2}{2}\int\limits_{r_{0}}^{r_{1}} \frac{3\rho^{2}\green_r(\rho)}{(1-\rho^{2})^3} \, d\rho - C_1\int\limits_{0}^{r_{0}} \frac{3\rho^{2}\green_r(\rho)}{(1-\rho^{2})^3} \, d\rho,
\end{equation}

Fix  $r_1$, and let $r \to 1$.  Since $\green_r(\rho) \to \green(\rho)$ uniformly on compact sets in $[0,1)$, it follows that
\begin{equation}\label{eq-3}
\int\limits_{r_{0}}^{r_{1}} \frac{3\rho^{2}\green_r(\rho)}{(1-\rho^{2})^3} \, d\rho \to \int\limits_{r_{0}}^{r_{1}} \frac{3\rho^{2}\green(\rho)}{(1-\rho^{2})^3} \, d\rho ,  \,\,\,  r \to 1.
\end{equation}
Using  (\ref{eq-green-est}), we get
\begin{align*}
\liminf\limits_{r \to 1} \int\limits_{r_{0}}^{r_{1}} \frac{3\rho^{2}\green_r(\rho)}{(1-\rho^{2})^3} \, d\rho &\ge  \const_2 \int\limits_{r_{0}}^{r_{1}} \frac{3\rho}{(1-\rho^{2})} \, d\rho \\
&= \frac{3\const_2}{2}\int\limits_{r^{2}_{0}}^{r^{2}_{1}} \frac{1}{1-\rho} \, d\rho \\ 
&=\frac{3\const_2}{2} \log\frac{(1-r^{2}_{0})}{(1-r^{2}_{1})} \to \infty, \,\,r_1 \to 1,
\end{align*}
which  together with (\ref{eq-2}) and (\ref{eq-3})  proves (\ref{eq-formula-3}).

\vskip .3cm

We prove the lemma by contradiction. Suppose that there exists a sequence of $K$-quasiconformal maps $f_N \in \QC(\Sp^{2})$, $N \in \N$, such that

\begin{equation}\label{eq-formula-1}
\int\limits_{\Bnn} \green_r (x)\Phi(f_N)(x) \, d\lambda(x) \le M,
\end{equation}
for every $0<r\le 1-\frac{1}{N}$. 

The maps $f_N$ are normalized and $K$-quasiconformal. Thus, after passing onto a subsequence if necessary,  there exists a $K$-quasiconformal map $f$ such that $f_N \to f$ pointwise on $\Sp^{2}$.

If $x \in X_f(\epsilon_0)$, then  $x \in X_{f_{N}}(\epsilon_0)$ for $N$ sufficiently large. This follows from the the continuity of $\good$, and the fact that $X_f$ is an open set. It follows from the definition of $\Phi(f)$ that for every $x \in \Bnn$, we have
$$
\Phi(f)(x)\le \Phi(f_N)(x),\,\,\text{for N large enough}.
$$
This implies (using the Fatou Lemma) that for a fixed $0 \le r<1$, we have
$$
\liminf_{N \to \infty} \int\limits_{\Bnn} \green_{r} (x)\Phi(f_N)(x) \, d\lambda(x) \ge   \int\limits_{\Bnn} \green_{r} (x)\Phi(f)(x) \, d\lambda(x),
$$
which together with (\ref{eq-formula-1}) implies that the estimate
\begin{equation}\label{eq-formula-2}
\int\limits_{\Bnn} \green_{r} (x)\Phi(f)(x) \, d\lambda(x) \le M,
\end{equation}
holds for every $0\le r<1$. But this contradicts (\ref{eq-formula-3}) and we are finished.

\end{proof}

\section{Proof of Theorem \ref {thm-main-1}} We prove the theorem assuming Theorem \ref{thm-extension}. The theorem essentially follows from  Lemma \ref{lemma-estimate}. The proof of the lemma 
relies on the uniformity property of $\good$ we proved above  and Lemma \ref{lemma-measure} we prove below.

\subsection{The distance between $\good$ and the harmonic extension} If a map $f \in \QC(\Sp^{2})$ admits a harmonic quasi-isometric extension, we denote this extension by $H(f):\Bnn \to \Bnn$ (such $H(f)$ is unique by \cite{l-w}).  We let
$$
\dista(f)(x)= d_{\Ho}\big(\good(f)(x),H(f)(x) \big).
$$
By $||\dista(f)||$ we denote the supremum of $\dista(f)$ over $\Bnn$, which is finite since both $\good(f)$ a $H(f)$ are quasi-isometries.

To prove Theorem \ref{thm-main-1}, we need to show that  if $f$ is $K$-quasiconformal, then $H(f)$ is $(L_1,A_1)$-quasi-isometry, for some constants $L_1=L_1(K)$ and $A_1=A_1(K)$. 
To achieve this, it is sufficient to prove that 
\begin{equation}\label{eq-new}
||\dista(f)|| \le D_1,
\end{equation}
for some constant $D_1=D_1(K)$. We establish this inequality in the remainder of this section.

For each $f \in \QC(\Sp^{2})$, we choose a point $x \in \Bnn$ where $\dista(f)(x)> ||\dista(f)||-1$. Let $I \in \Isom(\Ho)$ be such that $I(0)=x$, and  $J \in \Isom(\Ho)$ be such that $J \circ f \circ I$  is normalized to fix the three points on $\Sp^{2}$.

By Lemma \ref{lemma-distance-bound}, the distance between $\good(J\circ f \circ I)$ and $J\circ \good(f) \circ I$ is uniformly bounded above by $D=D(K)$. Since $H(J\circ f \circ I)=J\circ H(f) \circ I$,  we obtain the estimate
$$
\dista(J\circ f \circ I)(0)>  ||\dista(J\circ f \circ I)||_{\infty}-D-1.
$$
\vskip .3cm
Thus, after replacing $f$ by $J\circ f \circ I$ if necessary, we may assume that $f$ is such that
\begin{equation}\label{eq-norm}
\dista(f)(0)>  ||\dista(f)|| -D-1.
\end{equation}
In the rest of the section we prove (\ref{eq-new}) for every $f$ having this property (which will prove Theorem \ref{thm-main-1}).

From now on, we assume that (\ref{eq-norm}) holds.

\subsection{The main formula for $\dista^2(f)$} 
The proof of inequality (\ref{eq-new}) for a $K$-quasiconformal map $f$ satisfying (\ref{eq-norm}) (and thus the proof of Theorem \ref{thm-main-1}) is based on the analysis of the following identity
\vskip .3cm
\begin{equation}\label{eq-osnovna-00}
\dista^2(f)(0) +\int\limits_{\Bnn} \green_r(x) \lap \dista^2(f)(x)\, d\lambda(x) = \Avr_{\dista^{2}(f)}(r) \le ||\dista(f)||^{2},
\end{equation}
\vskip .3cm
\noindent
for every $0\le r<1$. 

Since $\dista^2(f)$ is $C^2$ on $\Bnn$, this formula follows from  (\ref{eq-green}) in Lemma \ref{lemma-green} (the upper bound in (\ref{eq-osnovna-00}) follows from the assumption that $\sigma$ is a probability measure on $\Sp^2$).

Replacing (\ref{eq-norm}) in (\ref{eq-osnovna-00}), we obtain the following inequality
\vskip .3cm
\begin{align}\label{eq-osnovna-1}
||\dista(f)||^{2}-D' ||\dista(f)||-D'' + &\int\limits_{\Bnn} \green_r(x) \lap \dista^2(f)(x)\, d\lambda(x) \le \\
\nonumber &\le \Avr_{\dista^{2}(f)}(r) \le ||\dista(f)||^{2},
\end{align}
\vskip .3cm
\noindent
where $D'=D'(K)=2(D-1)$ and $D''=D''(K)=(D+1)^2$.

\subsection{Estimating the set where $\dista^2(f)$ is small} Set
\begin{equation}\label{eq-Y}
Y_f=\{x \in \Bnn : \, \dista(f)(x) \ge \frac{1}{2}||\dista(f)|| \}.
\end{equation}

We now bound from below the measure of the set $r\Sp^2 \setminus Y_f$.

\begin{lemma}\label{lemma-measure} For each $K \ge 1$, there exists an increasing  function $\varphi_K(r)$, $0\le r<1$, such that assuming $||\dista^2(f)||\ge 1$,  the estimate 
$$
\sigma\big( \{\zeta \in \Sp^{2}:\, r\zeta \notin Y_{f} \}\big) \le \frac{\varphi_K(r)}{ ||\dista(f)||},
$$
holds for every $K$-quasiconformal $f \in \QC(\Sp^2)$ satisfying the inequality (\ref{eq-norm}).
\end{lemma} 

\begin{proof} For simplicity, set  
$$
Y_f(r)=\{\zeta \in \Sp^{2}:\, r\zeta \in Y_{f} \}.
$$
Thus, we need to show 
$$
\sigma\big(\Sp^2 \setminus Y_f(r)\big) \le \frac{\varphi_K(r)}{ ||\dista(f)||}.
$$

We now use the lower bound (\ref{distance-1}) 
$$
\lap \dista^2(f)(x) \ge -2\dista(f)(x)|\tau(\good(f)(x)| \ge -2||\dista(f)|| ||\tau(\good(f))||,\,\, x\in \Bnn,
$$
from Lemma \ref{lemma-distance} to bound below the Laplacian $\lap \dista^2(f)$.  Since $||\tau(\good(f)(x)||\le T$ (recall that $T$ is the constant from Definition \ref{def-admissible}), we obtain from the left hand side inequality in (\ref{eq-osnovna-1})
\begin{align*}
&||\dista(f)||^{2}\left(1-\frac{D'}{||\dista(f)||}-\frac{D''}{||\dista(f)||^2} -\frac{2T}{||\dista(f)||} \int\limits_{\Bnn} \green_r(x)\, d\lambda(x)  \right) \le \\
&\le \Avr_{\dista^{2}(f)}(r) = \int\limits_{Y_{f}(r)}  \dista^2(f)(r\zeta)\, d\sigma(\zeta) + \int\limits_{\Sp^2 \setminus Y_{f}(r)}  \dista^{2}(f)(r\zeta)\, d\sigma(\zeta) \le \\
&\le \int\limits_{Y_{f}(r)}  ||\dista(f)||^2\, d\sigma + \int\limits_{\Sp^2 \setminus Y_{f}(r)}  \frac{1}{4}||\dista(f)||^2\, d\sigma.
\end{align*}
Dividing both sides by $||\dista(f)||^{2}$, and using the identity $1=\sigma(Y_f(r))+\sigma(\Sp^2 \setminus Y_f(r))$, we obtain the inequality
$$
\sigma(\Sp^2 \setminus Y_f(r))\le \frac{4}{3||\dista(f)||} \left( D'+\frac{D''}{||\dista(f)||} +2T \int\limits_{\Bnn} \green_r(x)\, d\lambda(x) \right), 
$$
which proves the lemma for 
$$
\varphi_K(r)=\frac{4}{3}\left( D'+D'' +2T \int\limits_{\Bnn} \green_r(x)\, d\lambda(x) \right).
$$
(Here we used the assumption $||\dista(f)||\ge 1$). The function $\varphi_K(r)$ is increasing.

\end{proof}

\subsection{The main lemma} Theorem \ref{thm-main-1} essentially follows by applying the following lemma to the main inequality (\ref{eq-osnovna-1}).

\begin{lemma}\label{lemma-estimate} For every $K \ge 1$, there exists an increasing function $\psi_K(r)$, $0\le r<1$, with the following properties. For every $M>0$, there exists $0<r_0=r_0(M,K)<1$, such that assuming $||\dista(f)||\ge 1$, the inequality 
\begin{equation}\label{eq-osnovna}
\int\limits_{\Bnn} \green_{r_{1}}(x) \lap \dista^2(f)(x)\, d\lambda(x) \ge M||\dista(f)||-\psi_K(r_0),
\end{equation}
holds for some $0<r_1 \le r_0$, and every $K$-quasiconformal map $f$ satisfying (\ref{eq-norm}).
\end{lemma}

\begin{proof} Recall the set $X_f(\epsilon)$ defined in (\ref{eq-X}). We let 
$$
\epsilon_0=\epsilon_0(K)= \frac{1}{4}q \tanh \frac{1}{4},
$$
where $q=q(2K)$ is the constant from Lemma \ref{lemma-distance}. Set
\begin{equation}\label{X-1}
X_f=X_f(\epsilon_0).
\end{equation}

From the inequality (\ref{distance-2}) in  Lemma \ref{lemma-distance},  and since $ \dista(f)(x) \ge \frac{1}{2}||\dista(f)||\ge \frac{1}{2}$, $x \in Y_f$, we get
$$
\lap \dista^2(f)(x) \ge -2\dista(f)(x) \epsilon_0 + 2q\dista(f)(x)  \tanh \frac{1}{4}, \,\, \, x \in  X_f\cap Y_f. 
$$
(Here we use  that by definition $|\tau(\good(f))(x)|\le \epsilon_0$ for $x \in X_f$.) Thus, by the choice of $\epsilon_0$, and since $ \dista(f)(x) \ge \frac{1}{2}||\dista(f)||$, $x \in Y_f$,   we get
$$
\lap \dista^2(f)(x) \ge -2||\dista(f)||\frac{1}{4}q \tanh \frac{1}{4}  + 2q\frac{1}{2}||\dista(f)||  \tanh \frac{1}{4}, \,\,\, x \in X_f\cap Y_f,
$$
which yields
\begin{equation}\label{eq-good} 
\lap \dista^2(f)(x) \ge \frac{1}{2}q||\dista(f)|| \tanh \frac{1}{4}, \,\, \, x \in X_f\cap Y_f. 
\end{equation}

As we already pointed out above, we have 
\begin{equation}\label{eq-bad} 
\lap \dista^2(f)(x) \ge -2||\dista(f)||T, \, \,  x \in \Bnn,
\end{equation}
where $T=T(K)$ is the constant from Definition \ref{def-admissible}. 

\vskip .3cm
Next,  we define the function $\Phi(f)=\Phi(f,C_1,C_2)$ (from Lemma \ref{lemma-pomocna}) by 
$$
\Phi(f)(x)=
\begin{cases}
\frac{1}{2} q \tanh \frac{1}{4}, \,\, \text{if} \,\, x \in X_f, \\
-2T,\,\,  x\in \Bnn \setminus X_f,
\end{cases}
$$
where $C_2= C_2(K)=\frac{1}{2} q \tanh \frac{1}{4}$ and $C_1=C_1(K)=2T$. (We define $\Phi(f)$ for every $f \in \QC(\Sp^{2})$ regardless of whether it has the harmonic extension $H(f)$.)

It follows from (\ref{eq-good})  that
$$
\lap \dista^2(f)(x) \ge ||\dista(f)|| \Phi(f)(x), \,\,\, \text{for}\,\,\, x \in X_f  \cap Y_f.
$$
On the other hand, from (\ref{eq-bad})  it follows  that
$$
\lap \dista^2(f)(x) \ge ||\dista(f)|| \Phi(f)(x), \,\,\, \text{for}\,\,\,  x \in \Bnn \setminus X_f.
$$
Finally, we estimate from below $\lap \dista^2(f)(x)$ on $X_f \setminus Y_f$,. Here we use the estimate that holds on the entire $\Bnn$ (which  follows from (\ref{eq-bad}) and the definition of $\Phi(f)$)
$$
\lap \dista^2(f)(x)-||\dista(f)||\Phi(f)(x) \ge -P||\dista(f)||,\,\, x \in \Bnn,
$$
where 
$$
P=2T+ \frac{1}{2} q \tanh \frac{1}{4}.
$$

The last three inequalities yield

\begin{align}\label{eq-nova-2}
&\int\limits_{\Bnn} \green_r(x) \lap \dista^2(f)(x)\, d\lambda(x)- ||\dista(f)|| \int\limits_{\Bnn} \green_r(x) \Phi(f)(x) \, d\lambda(x) \ge \\ 
\nonumber &\ge -P||\dista(f)|| \int\limits_{\Bnn\setminus Y_f} \green_r(x)\, d\lambda(x),\,\, 0\le r<1. 
\end{align}
By Lemma \ref{lemma-average}, we have
$$
\int\limits_{\Bnn\setminus Y_f} \green_r(x)\, d\lambda(x)=\int\limits_{0}^{r} \green_r(\rho) \Avr_{\chi}(\rho)\, d\rho(x),
$$
where $\chi$ is the characteristic function of the set $\Bnn \setminus Y_f$. Thus, by Lemma \ref{lemma-measure} we have the bound
\begin{equation}\label{eq-nova-11}
\int\limits_{\Bnn\setminus Y_f} \green_r(x)\, d\lambda(x) \le \frac{\varphi_K(r)}{ ||\dista(f)||} \int\limits_{0}^{r} \green_r(\rho)\, d\rho(x). 
\end{equation}

Let
$$
\psi_K(r)=P\varphi_K(r) \int\limits_{0}^{r} \green_r(\rho)\, d\rho(x),
$$
where $\varphi_K(r)$ is the increasing  function from Lemma \ref{lemma-measure}. Therefore, the function $\psi_K(r)$ is also increasing. 
Replacing (\ref{eq-nova-11}) in (\ref{eq-nova-2}) we get 
\begin{equation}\label{eq-nova-2-new}
\int\limits_{\Bnn} \green_r(x) \lap \dista^2(f)(x)\, d\lambda(x)- ||\dista(f)|| \int\limits_{\Bnn} \green_r(x) \Phi(f)(x) \, d\lambda(x) \ge -\psi_K(r).
\end{equation}

\vskip .3cm

On the other hand, from Lemma \ref{lemma-pomocna}, for every $M>0$ we can find  $r_0=r_0(M,K)$ such that 
\begin{equation}\label{eq-formula-10}
\int\limits_{\Bnn} \green_{r_{1}}(x)\Phi(f)(x) \, d\lambda(x)\ge M,
\end{equation}
for some $0< r_1 \le r_0$. Replacing (\ref{eq-formula-10}) in  (\ref{eq-nova-2-new}), and using that $\psi_K(r_0)\ge \psi_K(r_1)$,  proves the lemma.

\end{proof}

\subsection{The Endgame: A proof of (\ref{eq-new})} 
We now prove the estimate (\ref{eq-new}) for every $K$-quasiconformal map satisfying (\ref{eq-norm}), thus proving Theorem \ref{thm-main-1}.

From (\ref{eq-osnovna-1}) we get 
$$
||\dista(f)||^{2}-D' ||\dista(f)||-D'' + \int\limits_{\Bnn} \green_r(x) \lap \dista^2(f)(x)\, d\lambda(x) \le ||\dista(f)||^{2},
$$
and by subtracting $||\dista(f)||^2$ from both sides, we obtain
\begin{equation}\label{eq-kraj}
\int\limits_{\Bnn} \green_r(x) \lap \dista^2(f)(x)\, d\lambda(x) \le D'||\dista(f)||+D'' ,
\end{equation}
for every $0\le r<1$. 

Assume $||\dista(f)||\ge 1$ (otherwise we are done). Let $M=D'+1$, and let $r_0=r_0(K,M)=r_0(K)$ be the corresponding constant from Lemma \ref{lemma-estimate}.  Then
$$
\int\limits_{\Bnn} \green_{r_{1}}(x) \lap \dista^2(f)(x)\, d\lambda(x) \ge M||\dista(f)||-\psi_K(r_0)=(D'+1)||\dista(f)||-\psi_K(r_0),
$$
holds for some $0<r_1\le r_0$. Replacing this in (\ref{eq-kraj}), we get
$$
(D'+1)||\dista(f)||-\psi_K(r_0)\le D'||\dista(f)||+D'',
$$
which yields
$$
||\dista(f)|| \le \psi_K(r_0)+D''=D_1=D_1(K).
$$
This proves (\ref{eq-new}) and completes the proof of Theorem \ref{thm-main-1}.

\section{Constructing the $\good$-extension}  We prove Theorem \ref{thm-extension} by explicitly constructing $\good$.  By $\Hoo\equiv \R^2 \times (0,\infty)$, we denote the upper half subspace of $\R^3$. Every point $z \in \Hoo$ has coordinates $z=(x,t)$, where $x=(x_1,x_2) \in \R^2$ and $t>0$. We consider $\Hoo$ as  the upper-half space model of the hyperbolic 3-space $\Ho$.  By $\Isom_{\infty}(\Hoo)$ we denote all M\"obius maps of $\Hoo$ that fix $\infty$.

We let $\QC(\R^2)$ stand for the space of quasiconformal maps of $\R^2$, fixing $(0,0),(1,0), \infty$.  We construct the extension 
$$\good:\QC(\R^2) \to \big( C^2(\Hoo)\cap \QI(\Hoo) \big). 
$$
The required extension of maps from $\QC(\Sp^2)$ is then obtained by conjugating $\good$ by the M\"obius transformation that maps $\overline{\R^2}$ to $\Sp^2$ and the points $(0,0),(1,0),\infty$ to the three points on $\Sp^2$ that are fixed by maps from $\QC(\Sp^2)$.

Every map $f \in \QC(\R^2)$ is differentiable almost everywhere (with the derivative of maximal rank). If $f(x)=(u_1(x_1,x_2),u_2(x_1,x_2))$, we let
$$
\ener(f)(x)=\sum_{i,j=1}^{2} \sum_{\alpha,\beta=1}^{2} \frac{\pt u_i}{\pt x_j}(x)\frac{\pt u_\alpha}{\pt x_\beta}(x),
$$
denote the energy density of $f$ (where defined) computed with respect to the Euclidean metric.

\subsection{Harmonic extensions of linear maps} Every invertible, orientation preserving linear map $L:\R^2 \to \R^2$ can be written as
$$
L(x)=(ax_1+bx_2,cx_1+dx_2),
$$
for some real numbers $a,b,c$ and $d$,  such that $ad-bc>0$. The space of such linear mappings is denoted by  $\Lin$. Normalized maps from  $\Lin$ are in $\QC(\R^2)$ and $\ener(L)$ and $\dist(L)$ are  constant functions on $\R^2$.

Every map  $L \in \Lin$ has a harmonic quasi-isometric extension $H(L)$ to $\Hoo$. It is given by (this was computed in \cite{l-t-1}, see also \cite{t-w})
\begin{equation}\label{eq-linear}
H(L)(x,t)=\left( L(x), \sqrt{\frac{\ener(L)}{2}} t \right).
\end{equation}

Since  the norm (with respect to the hyperbolic metric) of the vector $H_{*}(v_t)$, where $v_t$ is any unit vector parallel to $\frac{\partial}{\partial t}$,  is equal to 1,  it follows that 
\begin{equation}\label{eq-last-1}
\ener(H(L))(z)>1,\,\,\, z\in \Hoo, 
\end{equation}
for every $L \in \Lin$. Also  
\begin{equation}\label{eq-last-2}
\dist(H(L))(x,t)=\dist(L)(x),\,\,\, (x,t)\in \Hoo, 
\end{equation}
and every  $L \in \Lin$. Of course,
\begin{equation}\label{eq-last-3}
\tau(H(L))(x,t)=0,\,\,\, (x,t)\in \Hoo, 
\end{equation}
since $H(L)$ is harmonic.

\subsection{The definition of $\good$} We construct $\good$ using convolution operators, and $\good$ can be seen as 
 a generalization of the Ahlfors-Beurling extension (see \cite{k-o} for a very similar extension).

We let $\good(f):\Hoo \to \Hoo$ be given by
$$
\good(f)(x,t)=\left( \int\limits_{\R^2}f(x+ty)\phi(y)\, dy_1dy_2,\,  \frac{t}{\sqrt{2}} \sqrt{ \int\limits_{\R^2}\ener(f)(x+ty)\phi(y)\, dy_1dy_2 }  \right),
$$
where $\phi$ is the Gauss kernel
$$
\phi(y)=\frac{1}{2\pi} e^{-\frac{|y|^{2}}{2} },
$$
as in \cite{k-o}.  The two integrals used to construct $\good$ are   the Gauss-Weierstrass transformations of $f$ and $\ener(f)$ respectively.

Every $K$-qc map $f \in \QC(\R^2)$ is H\"older continous which implies that 
\begin{equation}\label{nnn}
|f(x)|\le const|x|^K,\,\,\, |x|>1.
\end{equation}
This shows that the first integral in the definition of $\good(f)$ converges. 

We now show that the second integral is well defined. Let 
$$
I(y)=I(y_1,y_2)=\big( \frac{y_1}{y_{1}^2+y_{2}^2}, \frac{-y_2}{y_{1}^2+y_{2}^2}     \big),
$$
be the inversion mapping. There exists a neighborhood of infinity $\Omega \subset \R^2$ such that  $I \circ f$ maps $\Omega$ to is a bounded neighborhood of $ 0 \in \R^2$. Thus the Euclidean area of $(I \circ f)(\Omega)$ is finite and we get
$$
\int\limits_{\Omega} J(I)(f(y)) J(f)(y) \, dy_1dy_2 <\infty.
$$
Since $J(I)(y)=\frac{1}{|y|^{4}}$, and using (\ref{nnn}), we obtain 
$$
\int\limits_{\Omega} \frac{ J(f)(y)}{|y|^{4K}} \, dy_1dy_2 < \int\limits_{\Omega} \frac{J(f)(y) }{|f(y)|^{4K}}  \, dy_1dy_2 < \infty.
$$
Since $f$ is quasiconformal, the energy 
$$
\ener(f) \asymp J(f),
$$ 
is comparable with the Jacobian of $f$ (uniformly on $\R^2$) and we get 
$$
\int\limits_{\Omega} \frac{ \ener(f)(y) }{|y|^{4K}}  \, dy_1dy_2  <  \infty.
$$
It follows that the second integral in the definition of $\good$ is well defined.
\vskip .3cm

In fact, it follows from standard facts about convolution operators that $\good(f)$ is $C^{\infty}$ on $\Hoo$ and that $\good$ is conformally natural with respect to $\Isom_{\infty}(\Hoo)$, 
that is $I\circ \good(f) \circ J=\good(I \circ f \circ J)$, for $I,J \in \Isom_{\infty}(\Hoo)$. 

Clearly $\good(f)$  extends continuously to $\R^2$ (where it agrees with $f$) (see \cite{k-o}). Since $\good(f)$ is smooth and conformally natural with respect to  the transitive group $\Isom_{\infty}(\Hoo)$, by the standard compactness argument and other basic properties  of quasiconformal maps (see  any book on quasiconformal mappings) we verify that $\good$ is an admissible extension in the sense of Definition \ref{def-admissible}. We leave to the reader to check this.

Moreover, if $L \in \Lin$, then from linearity of $L$ we get
$$
\int\limits_{\R^2}L(x+ty)\phi(y)\, dy_1dy_2=L(x).
$$
Since $\ener(L)$ is a constant function  on $\R^2$, it follows 
$$  
\sqrt{ \int\limits_{\R^2}\ener(L)(x+ty)\phi(y)\, dy_1dy_2  }=\sqrt{ \ener(L) }.
$$
Thus, $\good(L)=H(L)$ is harmonic on $\Hoo$.

\subsection{The proof of Theorem \ref{thm-extension} } 
It remains to prove (\ref{convergence}) from Theorem \ref{thm-extension}. Fix $\epsilon>0$ and for a $K$-qc map $f \in \QC(\R^2)$ we let 
$$
X_f(\epsilon)=\{z \in \Hoo:\, \ener(\good(f))(z)>1,\, \dist(\good(f))(z)<2K,\, |\tau(\good(f))(z)|<\epsilon \}.
$$
We have the following lemma.

\begin{lemma}\label{lemma-last} Suppose that $f$ is differentiable (with the derivative of maximal rank) at some point $x \in \R^2$. Then there exists $t_0=t_0(f,\epsilon)>0$ such that  
$(x,t) \in X_f$, for every $0<t<t_0$.
\end{lemma}
\begin{proof} Again, we use the standard compactness argument. The proof is by contradiction. 

Suppose that there exists a sequence $t_n \to 0$, such that $(x,t_n) \notin X_f(\epsilon)$. We find $I_n \in \Isom_{\infty}(\Hoo)$ such that $I_n(x,1)=(x,t_n)$, and $J_n  \in \Isom_{\infty}(\Hoo)$ such that $f_n=J_n \circ f \circ I_n$ is normalized and therefore belongs to $\QC(\R^2)$. Note that the energy, the norm of the tension field and the distortion of $\good(f)$ at $(x,t_n)$ is equal respectively to the  energy, the norm of the tension field and the distortion of $\good(f_n)$ at $(x,1)$.

Since $f$ is differentiable at $x$  (with the derivative of maximal rank), the sequence $f_n$ converges pointwise to an element $L \in \Lin$. 
Then $\good(f_n) \to \good(L)$ in $C^{\infty}$ norm (which means that all derivatives converge on compact sets in $\Hoo$). In particular, the energy, the norm of the tension field and the distortion of $\good(f_n)$ at $(x,1)$ converges respectively  to the  energy, the norm of the tension field and the distortion of $\good(L)$ at $(x,1)$. 

Since we assumed that $(x,t_n) \notin X_f$, it follows that at least one of the three equations  (\ref{eq-last-1}),  (\ref{eq-last-2}),  (\ref{eq-last-3}) does not hold for $L$. This is a contradiction  and we are finished.

\end{proof}

\vskip .3cm

To finish the proof, we conjugate the extension $\good$ to the map (also denoted by) 
$$
\good:\QC(\Sp^2) \to \big( C^2(\Bnn)\cap \QI(\Bnn) \big). 
$$
We have already established that $\good$ is an admissible extension. From Lemma \ref{lemma-last}, it follows that if a $K$-qc map  $f\in \QC(\Sp^2)$ is differentiable (with the derivative of maximal rank) at some point $\zeta_0 \in \Sp^2$, then for some $0<r_0=r_0(f,\zeta_0, \epsilon)<1$, we have 
$r\zeta_0 \in X_f(\epsilon,2K)$, for every $r_0<r<1$. The formula (\ref{convergence}) from Theorem \ref{thm-extension}  now follows from Lebesgue's Dominated Convergence Theorem and we are finished.

\end{document}